\documentclass{paper}
\usepackage{amsthm}
\usepackage{amsmath}
\usepackage{amssymb}
\usepackage{xcolor}
\usepackage{tikz}

\usepackage[letterpaper, margin=1in]{geometry}

\newcommand{\Z}{\mathbb{Z}}

\author{Nick Ramsey \thanks{nramsey@depaul.edu}}
\institution{DePaul University}
\title{Perturbing Subshifts of Finite Type}

\begin{document}
\maketitle
\tableofcontents

\theoremstyle{plain}
\newtheorem{theorem}{Theorem}
\newtheorem{lemma}{Lemma}
\newtheorem{corollary}{Corollary}
\newtheorem{proposition}{Proposition}
\theoremstyle{definition}
\newtheorem{remark}{Remark}
\newtheorem{definition}{Definition}
\newtheorem{question}{Question}
\newtheorem{myexample}{Example}

\section{Introduction}
Let $n$ be a positive integer and let $T$ be an irreducible $n\times n$ matrix with entries in $\{0,1\}$.  This determines a subshift of finite type $\Sigma$ defined as the collection of all bi-infinite strings $(x_i)$ on an alphabet of $n$ symbols (indexing the rows and columns of $T$) that are {\it admissible} in the sense that  the $(x_{i+1},x_i)$ entry in $T$ is equal to $1$ for all $i\in\Z$.  The goal of this note is to explore the entropy $h(\Sigma)$ of the shift map on $\Sigma$ and how it is affected by perturbations obtained by forbidding various finite words from occurring in $\Sigma$. Recall that the largest eigenvalue of $T$ is given by $\lambda = e^{h(\Sigma)}$, and we assume throughout that $\lambda >1$. 

Let $S$ denote a finite set of finite admissible nonempty words none of which contains another, and let $\Sigma\langle S\rangle$ denote the subshift of $\Sigma$ consisting of those elements of $\Sigma$ in which none of the words in $S$ appear.  In \cite{lind}, Lind addressed the problem of bounding the entropy of $\Sigma\langle S \rangle$ in case $S$ consists of a single word. He proved that the entropy of $\Sigma\langle S \rangle$ approaches that of $\Sigma$ as the length $\ell$ of the word tends to infinity, and showed moreover that the difference in entropy is of order $\lambda^{-\ell}$. In \cite{psft}, the author adapted Lind's method to the case where $S$ has more than one word, and in particular introduced a certain determinant of correlation polynomials whose size is closely tied to the entropy perturbation.  Analyzing the size of this determinant gets complicated as $S$ grows, and we were only able to effectively bound the entropy and show that it approximates that of $\Sigma$ well in case $S$ consists of two words of  length tending to infinity. The entropy perturbation in this case is shown to be of order at most $\lambda^{-\ell}$ where $\ell$ is the length of the shorter word. 

In \cite{miller}, Miller introduced a different approach to the following related problem: given a finite set $S$ of finite nonempty words in an alphabet, determine whether there exists a bi-infinite word in the alphabet that avoids $S$.  Note that the ambient shift here is constrained to the full shift, while the set $S$ is quite flexible.  In this note, we adapt Miller's method to a general subshift $\Sigma$ and refine it to get lower bounds on the entropy of the perturbations $\Sigma\langle S\rangle.$  As in \cite{miller}, we define $$p(t) = \sum_{\tau\in S}t^{|\tau|}$$ 

\begin{theorem}\label{main}
	There exists a constant $C$ depending only on $\Sigma$ such that if $k$ is a positive integer, every element of $S$ has length at least $k$, and there exists $t\in (1,\lambda^k)$ with $$r = \frac{1+kC\lambda^{2k}p(t^{1/k}/\lambda)}{t} <1$$ then $$h(\Sigma) - h(\Sigma\langle S\rangle) \leq -\frac{\log(1-r)}{k}$$
\end{theorem}

This theorem is shown to have consequences for Lind-type problems. A perturbation bound of $O(\lambda^{-\ell})$ seems beyond the method, but we can establish $O(\ell^{-3/4}\lambda^{-\ell/8})$ when $S$ consists of any fixed number of words of minimal length $\ell$. The method can also be applied to a growing set of words of increasing length with sufficient control over the growth, as illustrated by the following result. Suppose that $S_1, S_2, \dots$ is a sequence of sets as above and let $\ell_i$ denote the minimal length in $S_i$. 
\begin{theorem}\label{notmain}
	Suppose there exists $\kappa<\lambda$ such that $|S_i| = O(\kappa^{\ell_i})$ as $i\to \infty$. Then $$h(\Sigma) = \lim_{i\to\infty}h(\Sigma\langle S_i\rangle)$$
\end{theorem}

\section{Parry measure and the weight function $w(\sigma)$}

Let $u$ and $v$ denote left and right $\lambda$-eigenvectors for $T$ normalized so that $u^t v=1$. The entries of these vectors measure the prominence of the corresponding symbols as a sink and source in $\Sigma$, respectively. More precisely, $u_i/\sum u_i$ is the fraction of paths on the directed graph associated to $T$ that terminate at $i$, while $v_i/\sum v_i$ is the fraction of paths that begin at $i$. Let $\mu$ denote Parry measure on $\Sigma$. This is a shift-invariant measure of maximal entropy and can be characterized on cylinder sets by $$\mu([ix_1x_2\cdots x_{k-1}j]) = \frac{u_iv_j}{\lambda^k}$$ The notation $[\sigma]$ only defines a cylinder set up to shifts. The shift-independence of $\mu$ often, as here, renders this ambiguity moot. Where it is important to have an actual set to work with (\emph{e.g.} in defining $f_\sigma$ below) we take $\sigma$ to begin at coordinate $0$ in forming $[\sigma]$. If $\sigma$ fails to be admissible, then $[\sigma]$ is taken to be empty. 

Let $N(\sigma,k)$ denote the number of words $\eta$ of length $k$ such that $\sigma\eta$ is admissible. Since $T$ is irreducible, there exist positive constants $A, B$ such that 
\begin{equation}\label{wordbound}
A\lambda^k\leq N(\sigma,k)\leq B\lambda^k
\end{equation}
 for all words $\sigma$ and all $k\geq 1$. Let $D$ denote the maximum ratio among the $v_i$.

\begin{lemma}\label{mulemma}
We have $$\mu([\sigma\tau])\leq DA^{-1}\lambda^{-|\tau|}\mu([\sigma])$$ for all words $\sigma, \tau$. 
\end{lemma}
\begin{proof}
The explicit description of $\mu$ on cylinder sets implies that $D$ is also the maximum ratio among the $\mu([\sigma\eta])$ as $\eta$ varies among words of a given positive length. Thus
\begin{eqnarray*}
	\lambda^{|\tau|}\mu([\sigma\tau]) \leq A^{-1}N(\sigma,|\tau|)\mu([\sigma\tau]) &=& A^{-1}\sum_{|\eta| =|\tau|}\mu([\sigma\tau]) \\ & \leq& A^{-1}\sum_{|\eta| = |\tau|}D\mu([\sigma\eta]) \\ &=& DA^{-1}\mu([\sigma])
\end{eqnarray*}
\end{proof}

Let $\sigma$ be an admissible word and let $[\sigma]$ denote the associated cylinder set. Define a polynomial-valued function $f_\sigma$ on $[\sigma]$ by $$f_\sigma(\alpha) =  \sum_{\tau\in S}\sum_j t^j$$ where the inner sum is over $j\geq 0$ such that $\tau$ occurs in $\alpha$ beginning within $\sigma$ and ending $j$ symbols beyond the end of $\sigma$. Observe that the function $f_\sigma$ is locally constant on $[\sigma]$. 
\bigskip
\begin{center}
\begin{tikzpicture}
	\def\sc{.4}
	\def\bs{0.8*\sc}
	\def\ss{0.2*\sc}
	\draw (-\sc,0) rectangle (-\sc+\bs,\bs);	
	\draw (0,0) rectangle (\bs,\bs);
	\draw (\sc,0) rectangle (\sc+\bs,\bs);
	\draw (2*\sc,0) rectangle (2*\sc+\bs,\bs);	
	\draw (3*\sc,0) rectangle (3*\sc+\bs,\bs);	
	\draw (4*\sc,0) rectangle (4*\sc+\bs,\bs);	
	\draw (5*\sc,0) rectangle (5*\sc+\bs,\bs);	
	\draw (6*\sc,0) rectangle (6*\sc+\bs,\bs);	
	\draw (7*\sc,0) rectangle (7*\sc+\bs,\bs);		
	\draw (0,-.1) -- (0,-.2) -- (4*\sc-\ss,-.2) -- (4*\sc-\ss,-.1);
	\draw (2*\sc,.5) -- (2*\sc,.6) -- (6*\sc+\bs,.6) -- (6*\sc+\bs,.5);
	\draw (.8,-.4) node {$\sigma$};
	\draw (1.8,.8) node {$\tau$};
	\draw (2.2,-.4) node {$\underbrace{\hspace{30pt}}_j$};
	\draw (-.7,.16) node {$\cdots$};
	\draw (3.5,.16) node {$\cdots$};
	\draw (-1.5,.17) node {$\alpha=$};
\end{tikzpicture}
\end{center}
\bigskip
We define a weight function on admissible words by $$w(\sigma) = \frac{1}{\mu([\sigma])}\int_{[\sigma]}f_\sigma$$ Note that the empty word $\sigma_0$ has cylinder set $[\sigma_0]=\Sigma$ and weight $0$. Since no element of $S$ contains another, for each $j\geq 0$ there is at most one element of $S$ that ends $j$ symbols after then of $\sigma$, so we may write 
\begin{equation}\label{explicit}
w(\sigma) = \sum_{j\geq 0} \frac{\mu(S_{\sigma,j})}{\mu([\sigma])} t^j
\end{equation}
where $S_{\sigma,j}$ denotes the subset of $[\sigma]$ containing an element $\tau\in S$ that begins in $\sigma$ and ends $j$ symbols after the end of $\sigma$. Observe that if $\sigma$ ends in an element of $S$, then we have $S_{\sigma,0}=[\sigma]$ and hence $w(\sigma)\geq 1$. In general, the weight $w(\sigma)$ is a measure of how close $\sigma$ is to ending in an element of $S$. The strategy here is study how $w$ changes as you extend $\sigma$ to the right by computing its weighted averages, and then use the results to bound from below the number of $S$-free extensions of $\sigma$ and ultimately the entropy of $\Sigma\langle S\rangle$. 

Define $$p_\sigma = \sum_{\tau\in S}\frac{\mu([\sigma\tau])}{\mu([\sigma])}t^{|\tau|}$$ Lemma \ref{mulemma} furnishes the upper bound \begin{equation}\label{pestimate}
p_\sigma \leq DA^{-1}\sum_{\tau\in S}(t/\lambda)^{|\tau|} = DA^{-1}p(t/\lambda)
\end{equation}
which is independent of $\sigma$. 

\begin{lemma}
	If $\sigma$ does not end in an element of $S$, then $$\frac{1}{\mu([\sigma])}\sum_i \mu([\sigma i])w(\sigma i) = \frac{w(\sigma)+p_\sigma}{t}$$
\end{lemma}
\begin{proof}
	Using (\ref{explicit}), 
	\begin{equation}\label{eq1}
	\frac{1}{\mu([\sigma])}\sum_i\mu([\sigma i])w(\sigma i) = \frac{1}{\mu([\sigma])}\sum_i\sum_j\mu(S_{\sigma i,j})t^j
	\end{equation}
	 An element of $S_{\sigma i,j}$ has a unique $\tau\in S$ ending $j$ symbols after $\sigma i$ and beginning within $\sigma i$. This $\tau$ can either begin within $\sigma$ or begin at the final symbol $i$, and accordingly we may decompose $S_{\sigma i,j} = A_{\sigma i,j}\sqcup B_{\sigma i,j}$. Now $$\bigsqcup_i A_{\sigma i,j} = S_{\sigma,j+1}$$ and $$\bigsqcup_iB_{\sigma i,j} = \bigsqcup_{\substack{\tau\in S\\|\tau|=j+1}}[\sigma\tau]$$ are both clear from the definitions. Thus (\ref{eq1}) is equal to  $$\sum_j\left(\frac{\mu(S_{\sigma,j+1})}{\mu([\sigma])}+\sum_{\substack{\tau\in S\\|\tau|=j+1}}\frac{\mu([\sigma\tau])}{\mu([\sigma])}\right)t^j = \frac{w(\sigma)+p_\sigma}{t}$$ Note that the last equality relies on the fact that $\sigma$ does not end in an element of $S$, so the apparently missing $\mu(S_{\sigma,0})$ in the sum on the left vanishes. 
\end{proof}

\section{Bounding entropy}

Fix some $t>1$ for the moment and let $\sigma$ be an $S$-free word with $w(\sigma)<1$.  We say that a word $\eta$ is {\it good} if $\sigma\eta$ is admissible and every intermediate word between $\sigma$ and $\sigma\eta$ (inclusive) has weight $w<1$.  In particular, $\sigma\eta$ is $S$-free if $\eta$ is good, since words ending in an element of $S$ have weight $\geq 1$.  For a positive integer $m$, set $$G(\sigma,m) = \bigsqcup_{\substack{\eta\ \mathrm{good}\\ |\eta|=m}} [\sigma\eta]$$
\begin{lemma}\label{mainestimate}
Suppose $\frac{1+p_\rho}{t} < r<1$ for all words $\rho$, and let $\sigma$ be $S$-free with $w(\sigma)<1$. We have $$\frac{\mu(G(\sigma,m))}{\mu([\sigma])} \geq (1-r)^m$$ for all $m\geq 1$.
\end{lemma}
\begin{proof}
	Since extensions $\sigma i$ that end in an element of $S$ have $w(\sigma i)\geq 1$, we have $$\sum_i\mu([\sigma i])w(\sigma i) \geq \mu([\sigma])-\mu(G(\sigma,1))$$ Thus $$\frac{\mu(G(\sigma,1))}{\mu([\sigma])}\geq 1-\frac{1}{\mu([\sigma])}\sum_i\mu([\sigma i])w(\sigma i)=1 - \left(\frac{w(\sigma)+p_\sigma}{t}\right) \geq 1-r$$ which establishes the case $m=1$. 
	
Suppose the statement holds for some $m\geq 1$ and all $\sigma$. Observe that $$G(\sigma,m+1) = \bigsqcup_{\substack{\eta\ \mathrm{good}\\ |\eta|=m}} G(\sigma\eta,1)$$ For good $\eta$, the word $\sigma\eta$ is $S$-free and has $w(\sigma\eta)<1$, so the base case and induction hypothesis give $$\mu(G(\sigma,m+1)) = \sum_{\substack{\eta\ \mathrm{good}\\|\eta|=m}} \mu(G(\sigma\eta,1)) \geq \sum_{\substack{\eta\ \mathrm{good}\\|\eta|=m}} (1-r)\mu([\sigma\eta]) = (1-r)G(\sigma,m)\geq (1-r)^{m+1}\mu([\sigma])$$ which establishes case $m+1$.
\end{proof}

\begin{proposition}\label{prop1}
	Suppose $r = \frac{1+DA^{-1}p(t/\lambda)}{t} < 1$. We have  $h(\Sigma) - h(\Sigma\langle S\rangle) < -\log(1-r)$
\end{proposition}
\begin{proof}
	Using (\ref{pestimate}), we may apply Lemma \ref{mainestimate} to the empty word $\sigma_0$ and conclude $\frac{\mu(G(\sigma_0,m))}{\mu([\sigma_0])}\geq (1-r)^m$. The set $[\sigma_0]$ is simply $\Sigma$, but we retain $\sigma_0$ below for clarity.  If $g$ denotes the number of good $\eta$ of length $m$, then we have $$\frac{G(\sigma_0,m)}{\mu([\sigma_0])}=\sum_{\eta\ \mathrm{good}}\frac{\mu([\sigma_0\eta])}{\mu([\sigma_0])}\leq gDA^{-1}\lambda^{-m}$$ by Lemma \ref{mulemma}. Thus we have produced for every $m\geq 1$ at least $$g\geq AD^{-1}\lambda^m(1-r)^m$$ words of length $m$ that are $S$-free, which implies that the entropy of $\Sigma\langle S\rangle$ is at least $$\lim_{m\to\infty} \frac{\log(AD^{-1}\lambda^m(1-r)^m)}{m} = \log(\lambda) +\log(1-r)$$ Since $\Sigma$ has entopy $\log(\lambda)$, this is the desired result. 
\end{proof}

\section{Blocking}

The condition $r = \frac{1+DA^{-1}p(t/\lambda)}{t} < 1$ in Proposition \ref{prop1} implies $t>1$ but also effectively limits $t$ from above to roughly $\lambda$. This in turn bounds $r$ from below, limiting the direct utility of Proposition \ref{prop1}. The solution is to work with blocks of elements in $\Sigma$. For each $k\geq 1$ let $\Sigma_k$ denote the SFT on the alphabet of admissible words of length $k$ in $\Sigma$, where the transition $[x_1\cdots x_k][y_1\cdots y_k]$ is admissible in $\Sigma_k$ if and only if $x_ky_1$ is admissible in $\Sigma$. Concatenating blocks furnishes a natural bijection $\Sigma_k\longrightarrow \Sigma$ that intertwines the shift map on $\Sigma_k$ with the $k$th power of the shift map on $\Sigma$. Accordingly, we have $h(\Sigma_k) = kh(\Sigma)$. 

To use the technique of the previous section, we must translate the collection $S$ of forbidden words into an equivalent collection $S_k$ of words in $\Sigma_k$ - that is, one that cuts out the same subshift under the above bijection. In the process, we will also bound the associated polynomial $$p_k(t) = \sum_{\zeta\in S_k}t^{|\zeta|}$$ Let $\tau\in S$ have length $\ell$, suppose that $k\leq \ell$, and write $\ell = kq+r$ according to the division algorithm. To determine a collection of words in $\Sigma_k$ that forbids $\tau$ in $\Sigma$, we must consider each of the $k$ ways of tiling over $\tau$ by blocks of length $k$, according to the $k$ possible positions of the beginning of $\tau$ in the first block. Of these $k$ positions, $r+1$ require $b=\lceil\ell/k\rceil$ blocks to tile over $\tau$. Here, $kb-\ell$ coordinates remain unspecified by $\tau$, which means that we have at most $B\lambda^{kb-\ell}$ words to consider at this position by (\ref{wordbound}). The remaining $k-r-1$ positions require $b+1$ blocks to tile over $\tau$ and leave $k(b+1)-\ell$ free coordinates. 
\bigskip
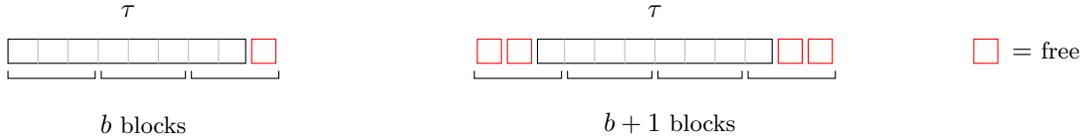
\begin{center}
\begin{tikzpicture}
	\def\sc{.4}
	\def\bs{0.8*\sc}
	\def\ss{0.2*\sc}
	\draw (0,0) rectangle (8*\sc-.5*\ss,\bs);
	\draw[color=red] (8*\sc+.5*\ss,0) rectangle (9*\sc-.5*\ss,\bs);
	\draw[color=lightgray] (1*\sc,0) -- (1*\sc,\bs);
	\draw[color=lightgray] (2*\sc,0) -- (2*\sc,\bs);
	\draw[color=lightgray] (3*\sc,0) -- (3*\sc,\bs);
	\draw[color=lightgray] (4*\sc,0) -- (4*\sc,\bs);
	\draw[color=lightgray] (5*\sc,0) -- (5*\sc,\bs);
	\draw[color=lightgray] (6*\sc,0) -- (6*\sc,\bs);
	\draw[color=lightgray] (7*\sc,0) -- (7*\sc,\bs);					
	\draw (0*\sc,-.1) -- (0*\sc,-.2) -- (3*\sc-.5*\ss,-.2) -- (3*\sc-.5*\ss,-.1);
	\draw (3*\sc+.5*\ss,-.1) -- (3*\sc+.5*\ss,-.2) -- (6*\sc-.5*\ss,-.2) -- (6*\sc-.5*\ss,-.1);
	\draw (6*\sc+.5*\ss,-.1) -- (6*\sc+.5*\ss,-.2) -- (9*\sc,-.2) -- (9*\sc,-.1);
	\draw (1.6,.7) node {$\tau$};	
	\draw (1.8,-.8) node {$b$ {\small blocks}};
	\def\os{7}
	\draw (\os+.5*\ss,0) rectangle (\os+8*\sc-.5*\ss,\bs);
	\draw[color=red] (\os+8*\sc+.5*\ss,0) rectangle (\os+9*\sc-.5*\ss,\bs);	
	\draw[color=red] (\os+9*\sc+.5*\ss,0) rectangle (\os+10*\sc-.5*\ss,\bs);		
	\draw[color=red] (\os-2*\sc+.5*\ss,0) rectangle (\os+-1*\sc-.5*\ss,\bs);		
	\draw[color=red] (\os-1*\sc+.5*\ss,0) rectangle (\os+0*\sc-.5*\ss,\bs);				
	\draw[color=lightgray] (\os+1*\sc,0) -- (\os+1*\sc,\bs);
	\draw[color=lightgray] (\os+2*\sc,0) -- (\os+2*\sc,\bs);
	\draw[color=lightgray] (\os+3*\sc,0) -- (\os+3*\sc,\bs);
	\draw[color=lightgray] (\os+4*\sc,0) -- (\os+4*\sc,\bs);
	\draw[color=lightgray] (\os+5*\sc,0) -- (\os+5*\sc,\bs);
	\draw[color=lightgray] (\os+6*\sc,0) -- (\os+6*\sc,\bs);
	\draw[color=lightgray] (\os+7*\sc,0) -- (\os+7*\sc,\bs);					
	\draw (\os+-2*\sc,-.1) -- (\os+-2*\sc,-.2) -- (\os+1*\sc-.5*\ss,-.2) -- (\os+1*\sc-.5*\ss,-.1);
	\draw (\os+1*\sc+.5*\ss,-.1) -- (\os+1*\sc+.5*\ss,-.2) -- (\os+4*\sc-.5*\ss,-.2) -- (\os+4*\sc-.5*\ss,-.1);
	\draw (\os+4*\sc+.5*\ss,-.1) -- (\os+4*\sc+.5*\ss,-.2) -- (\os+7*\sc-.5*\ss,-.2) -- (\os+7*\sc-.5*\ss,-.1);
	\draw (\os+7*\sc+.5*\ss,-.1) -- (\os+7*\sc+.5*\ss,-.2) -- (\os+10*\sc,-.2) -- (\os+10*\sc,-.1);
	\draw (\os+1.6,.7) node {$\tau$};
	\draw (\os+1.8,-.8) node {$b+1$ {\small blocks}};	
	\draw[color=red] (32*\sc+.5*\ss,0) rectangle (33*\sc-.5*\ss,\bs);
	\draw (34.5*\sc,.18) node {$=$ {\small free}};
\end{tikzpicture}
\end{center}
\bigskip
The total contribution to $p_k(t)$ of the words associated to $\tau$ is thus at most $$(r+1)B\lambda^{kb-\ell}t^b + (k-r-1)B\lambda^{k(b+1)-\ell}t^{b+1}$$ Assuming that $1\leq t\leq \lambda^k$, the contribution of $\tau$ to $p_k(t/\lambda^k)$ is then at most 
\begin{eqnarray*}
	(r+1)B\lambda^{-\ell}t^b+(k-r-1)B\lambda^{-\ell}t^{b+1} &\leq & (r+1)B\lambda^{k-\ell}t^{b-1}+(k-r-1)B\lambda^{2k-\ell}t^{b-1} \\ &\leq & kB\lambda^{2k-\ell}t^{\frac{\ell}{k}}
\end{eqnarray*}
Summing over $\tau\in S$, we have 
\begin{equation}\label{eq2}
p_k(t/\lambda^k) \leq kB\lambda^{2k}\sum_{\tau\in S}\left(\frac{t^{1/k}}{\lambda}\right)^{|\tau|} = kB\lambda^{2k}p\left(\frac{t^{1/k}}{\lambda}\right)
\end{equation}

\begin{proof}[Proof of Theorem \ref{main}]
The constants $A$, $B$, and $D$ depend on the underlying shift $\Sigma$ but do not change upon replacing $\Sigma$ by $\Sigma_k$. Set $C=DA^{-1}B$ and suppose that $$r =  \frac{1+kC\lambda^{2k}p(t^{1/k}/\lambda)}{t}<1$$ We may apply Lemma \ref{mainestimate} as in the previous section but now to $\Sigma_k$ to create at least $AD^{-1}\lambda^{km}(1-r)^m$ words of length $m$ in $\Sigma_k$, and thus words of length $km$ in $\Sigma$, that avoid $S$. The entropy of $\Sigma\langle S\rangle$ is therefore at least $$\lim_{m\to\infty}\frac{\log(AD^{-1}\lambda^{km}(1-r)^m)}{km} = \log(\lambda) + \frac{\log(1-r)}{k}$$  as desired. 
\end{proof}

\section{Growing words}

Let $\ell$ denote the minimal length of an element of $S$. Then  $$p(t^{1/k}/\lambda)\leq |S|(t^{1/k}/\lambda)^\ell$$ for $t\in(1,\lambda^k)$, so we consider $$r =  \frac{1+kC|S|\lambda^{2k-\ell}t^{\ell/k}}{t}$$ This function is minimized at $$t_\mathrm{min}=\lambda^{k}\frac{\lambda^{-2k^2/\ell}}{((\ell-k)C|S|)^{k/\ell}}$$ and has minimum $$r_\mathrm{min}=C^{k/\ell}|S|^{k/\ell}\ell^{k/\ell}\left(1-\frac{k}{\ell}\right)^{k/\ell-1} \lambda^{-k+2k^2/\ell}$$ Note that $t_\mathrm{min}\in (1,\lambda^k)$ as long as
\begin{equation}\label{criterion}
	1 < (\ell-k)C|S| < \lambda^{\ell-2k}
\end{equation}

Let $\alpha\in(0,1/2)$ and set $k = \lfloor \alpha\ell\rfloor$. Simple estimates show $$r_\mathrm{min} = O(|S|^\alpha\ell^\alpha\lambda^{-\ell\alpha(1-2\alpha)})$$ Since $\lambda^{\ell-2k}\geq \lambda^{\ell(1-2\alpha)}$, the condition (\ref{criterion}) is satisfied as long as 
\begin{equation}\label{criterion2}
1<(\ell-k)C|S|<\lambda^{\ell(1-2\alpha)}
\end{equation}

Suppose that $|S|$ is bounded as $\ell\to \infty$.  Then (\ref{criterion2}) holds for $\ell$ sufficiently large, and we have $r_\mathrm{min}\to 0$. Since $-\log(1-x) = O(x)$ for small $x$, Theorem \ref{main} gives $$h(\Sigma) - h(\Sigma\langle S\rangle) = O(\ell^{\alpha-1}\lambda^{-\ell\alpha(1-2\alpha)})$$ Setting $\alpha=1/4$ gives the best such bound, namely $O(\ell^{-3/4}\lambda^{-\ell/8})$, though we note that it is possible to improve this result slightly by using more refined estimates for $p_k$ in the previous section.

Now suppose that $|S|$ may be growing but subject to $|S|=O(\kappa^\ell)$ for some $\kappa<\lambda$. If we choose $\alpha$ small enough so that $\kappa<\lambda^{1-2\alpha}$, then condition (\ref{criterion2}) is satisfied for $\ell$ sufficiently large. We have $$r_\mathrm{min} = O(\kappa^{\alpha\ell}\ell^\alpha\lambda^{-\ell\alpha(1-2\alpha)}) = O\left( \left(\frac{\kappa}{\lambda^{1-2\alpha}}\right)^{\alpha \ell}\ell^\alpha\right)\to 0$$ as $\ell\to \infty$, so we may once again apply Theorem 1 to obtain $$h(\Sigma) - h(\Sigma\langle S\rangle) = O\left( \left(\frac{\kappa}{\lambda^{1-2\alpha}}\right)^{\alpha \ell}\ell^{\alpha-1}\right)\to 0$$ thereby establishing Theorem \ref{notmain}.

\bibliography{psftN}
\bibliographystyle{plain}

\end{document}